\documentclass[a4paper,reqno]{amsart}

\usepackage{graphicx}
\usepackage{latexsym,amssymb,amsmath,amscd,bm,nicefrac,a4wide,graphicx,enumerate,color}

\newtheorem{theorem}{Theorem}[section]
\newtheorem{algorithm}{Algorithm}
\newtheorem{lemma}[theorem]{Lemma}
\newtheorem{corollary}[theorem]{Corollary}
\newtheorem{proposition}[theorem]{Proposition} 
\theoremstyle{definition}

\theoremstyle{remark}
\newtheorem{remark}[theorem]{Remark}

\newcommand{\e}{\varepsilon}
\newcommand{\T}{\mathcal{T}}
\newcommand{\p}{\bm p}
\newcommand{\V}{V_{\rm hp}(\T,\p)}
\newcommand{\Vn}{V_{\rm hp}(\T^n,\p^n)}
\newcommand{\dd}{\mathsf{d}}
\newcommand{\dx}{\,\mathsf{d}x}
\renewcommand{\u}{u_{\rm hp}}

\newcommand{\N}[1]{\left|#1\right|}
\newcommand{\NN}[1]{\left\|#1\right\|}
\newcommand{\NNN}[1]{\left|\!\left|\!\left|#1\right|\!\right|\!\right|}
\newcommand{\err}{e_{\rm hp}}
\newcommand{\R}{\mathsf{R}_{\rm hp}[\err]}
\renewcommand{\P}{\Pi_{K_j}}
\newcommand{\jl}{[\![}
\newcommand{\jr}{]\!]}
\newcommand{\jmp}[1]{\jl#1\jr}
\newcommand{\intp}{\pi_{\V}}
\newcommand{\vhp}{v_{\rm hp}}

\title{A Posteriori Error Analysis of $hp$-FEM for singularly perturbed problems}

\author[J.~M.~Melenk]{Jens M.~Melenk}
\address{
J.~M.~Melenk, 
Institut f\"ur Analysis und Scientific Computing, 
TU Wien, Wiedner Hauptstrasse 8--10, 
A-1040 Wien, Austria.
}
\email{melenk@tuwien.ac.at}

\author[T.~P.~Wihler]{Thomas P.~Wihler}
\address{
T.~P.~Wihler,
Mathematics Institute, 
University of Bern, 
CH-3012 Bern, Switzerland.
}
\email{wihler@math.unibe.ch}
\thanks{TW acknowledges the support of the Swiss National Science Foundation.}



\begin{document}

\begin{abstract}
We consider the approximation of singularly perturbed linear second-order boundary value problems by $hp$-finite element methods. In particular, we include the case where the associated differential operator may not be coercive. Within this setting we derive an \emph{a posteriori} error estimate for a natural residual norm. The error bound is robust with respect to the perturbation parameter and fully explicit with respect to both the local mesh size~$h$ and the polynomial degree~$p$.
\end{abstract}

\maketitle

\section{Introduction}

%------
\emph{A posteriori} error estimation and adaptivity for low-order methods 
has seen a significant development in the last decades as witnessed 
by several monographs 
\cite{ainsworth-oden00,babuska-strouboulis01,verfurth13} on 
\emph{a posteriori} error estimation, and on convergence and optimality 
of adaptive algorithms; see, e.g., \cite{cascon-kreuzer-nochetto-siebert08,feischl-fuehrer-praetorius14,stevenson07}. 
The situation is less developed for high-order finite element methods 
($hp$-FEM), where both the local mesh size can be reduced and the local
approximation order can be increased to improve the accuracy. 

In an $hp$-context, several adaptive strategies and algorithms have been proposed (see~\cite{MiMc11} for an overview and comparison). The first work on $hp$-adaptive strategies for finite element approximations of elliptic problems was presented in~\cite{ROD89}. In addition, methods based on smoothness estimation techniques were proposed in~\cite{eibner-melenk06,HoustonSeniorSuliENUMATH,HoustonSuliHPADAPT,Ma94}, or in the recent approach~\cite{FaWiWi14,Wi11_2,Wi11} involving Sobolev embeddings, which will also be exploited in the present article. Moreover, a prediction technique was developed in~\cite{MelenkAPOST01}. Further $hp$-adaptive approaches in the literature include, for example, the use of \emph{a priori} knowledge, mesh optimization strategies, the Texas-3-step algorithm, or the application of reference solution
strategies; see, e.g., \cite{AiSe98,De07,DoHe07,GuiBabuska86,OdPaFe92}. Research focusing on the convergence of~$hp$-adaptive FEM has been developed only recently in~\cite{buerg-doerfler11,fourier}. 

In spite of the practical success of these $hp$-adaptive algorithms, 
\emph{a posteriori} error estimation in $hp$-FEM is still a topic of active 
research, and several, structurally different \emph{a posteriori} error 
estimators for $hp$-FEM for standard elliptic problems are available in 
the literature. We mention in particular the one of residual type, featuring 
a reliability-efficiency gap in the 
approximation order~\cite{doerfler-sauter13,MelenkAPOST01}, 
and the $p$-robust estimators of \cite{braess-schoeberl-pillwein09}, which is particularly suited for $H^1$-elliptic formulations.

Here, we present an \emph{a posteriori} error estimator 
for $hp$-FEM that is suitable for singularly perturbed problems; it is 
of residual type and results from merging the techniques of \cite{Ve1998} for 
singular perturbations with $p$-explicit estimators from \cite{MelenkAPOST01}. 
More precisely, on an interval~$\Omega=(a,b)\subset\mathbb{R}$, $a<b$, we consider
the singularly perturbed boundary value problem
\begin{alignat}{2}
-\e u''(x)+d(x)u(x)&=f(x),&\qquad x\in\Omega,\label{eq:1}\\
u(a)=u(b)&=0.\label{eq:2}
\end{alignat}
Here, $\e>0$ is a possibly small constant, $d\in L^\infty(\Omega)$ is
a given function, and~$f\in L^2(\Omega)$ is the right-hand side. We use standard notation: For an open set~$D\subseteq\Omega$, we let~$L^2(D)$ be the standard Lebesgue space of all square-integrable functions on~$D$ with norm~$\|\cdot\|_{L^2(D)}$, and~$L^\infty(D)$ is the space of all essentially bounded functions on~$D$ with norm~$L^\infty(D)$.

We
propose the following variational formulation
of~\eqref{eq:1}--\eqref{eq:2}: Find~$u\in H^1_0(\Omega)$, the standard
$L^2$-based Sobolev space of first order with vanishing trace, such
that
\begin{equation}\label{eq:var}
  a(u,v):=\e\int_\Omega u'(x)v'(x)\dx+\int_\Omega d(x)u(x)v(x)\dx=\int_\Omega f(x)v(x)\dx\qquad\forall v\in H^1_0(\Omega).
\end{equation}
Throughout this paper, we make the
general assumption that the solution of~\eqref{eq:var} exists and is
unique. Evidently, this the case if~$d\ge 0$.

The article is organized as follows: In the following Section~\ref{sc:hpFEM} we provide the $hp$-framework and $hp$-FEM for the discretization of~\eqref{eq:1}--\eqref{eq:2}. Furthermore, Section~\ref{sc:analysis} contains some $hp$-interpolation results, and the $hp$-\emph{a posteriori} error analysis. In addition, we present some numerical tests in Section~\ref{sc:numerics}. Finally, we summarize our work in Section~\ref{sc:concl}.

%%%%%%%%%%%%%%%%%%%%%%%%%%%%%%%%%%%%%%%%%%%%%%%%%%%%%%%%%%%%%%%%%%%

\section{$hp$-FEM Discretization}\label{sc:hpFEM}

In order to discretize the boundary value
problem~\eqref{eq:1}--\eqref{eq:2} by means of an $hp$-finite element
method, let us introduce a partition~$\T=\{K_j\}_{j=1}^N$ of~$N\ge 1$ (open)
elements~$K_j=(x_{j-1},x_j)$, $j=1,2,\ldots,N$ on~$\Omega=(a,b)$, with
\[
a=x_0<x_1<x_2<\ldots<x_{N-1}<x_N=b.
\]
The length of an element~$K_j$ is denoted by~$h_j=x_j-x_{j-1}$,
$j=1,2,\ldots,N$. For each element $K_j \in \T$, it will be convenient 
to introduce the patch $\widetilde K_j = \bigcup \{K_i\in\T\,|\, 
\overline{K_i}\cap \overline{K_j} \ne \emptyset\}$ as the union of~$K_j$ and
of the elements adjacent to it. 
In addition, to each element~$K_j$ we associate a
polynomial degree~$p_j\ge 1$, $j=1,2,\ldots,N$. These numbers are
stored in a polynomial degree vector~$\p=(p_1,p_2,\ldots,p_N)$. Then,
we define an $hp$-finite element space by
\[
\V=\left\{v\in H^1_0(\Omega):\,
v|_{K_j}\in\mathbb{P}_{p_j}(K_j),\,j=1,2,\ldots,N\right\},
\]
where, for~$p\ge 1$, we denote by~$\mathbb{P}_p$ the space of all
polynomials of degree at most~$p$. We say that the pair $(\T,\p)$ 
of a partition $\T$ and of a degree vector $\p$ is $\mu$-shape regular, for some constant~$\mu>0$ independent of~$j$, if 
\begin{equation}
\label{eq:gamma-shape-regular}
\mu^{-1} h_{j+1} \leq h_{j} \leq \mu h_{j+1}, 
\qquad  
\mu^{-1} p_{j+1} \leq p_{j} \leq \mu p_{j+1}, 
\qquad j=1,\ldots,N-1, 
\end{equation}
i.e., if both the element sizes and polynomial degrees of neighboring 
elements are comparable. 

We can now discretize the variational formulation~\eqref{eq:var} by
finding a numerical approximation~$\u\in\V$ such that
\begin{equation}\label{eq:hpFEM}
a(\u,v)=\int_\Omega fv\dx\qquad\forall v\in\V.
\end{equation}
As in the continuous case, we generally suppose that, for a given $hp$-space~$\V$, a unique numerical solution~$\u\in\V$ of~\eqref{eq:hpFEM}
exists.

Furthermore, let us introduce the following norm on~$H^1_0(\Omega)$:
\begin{equation}\label{eq:norm}
\NNN{v}^2:=\sum_{j=1}^N\NNN{v}_{K_j}^2
:=\sum_{j=1}^N\left(\e\NN{v'}^2_{L^2(K_j)}+\NN{\sqrt{|d|} v}^2_{L^2(K_j)}\right).
\end{equation}
We note that, if~$d\ge 0$ on~$\Omega$, then the norm~$\NNN{\cdot}$
equals the natural energy norm corresponding to the bilinear
form~$a(\cdot,\cdot)$ from~\eqref{eq:var}. More precisely, in that
case we have that~$a(v,v)=\NNN{v}^2$ for any~$v\in H^1_0(\Omega)$.

%%%%%%%%%%%%%%%%%%%%%%%%%%%%%%%%%%%%%%%%%%%%%%%%%%%%%%%%%%%%%%%%%%%

\section{Robust A Posteriori Error Analysis}\label{sc:analysis}

The goal of this section is to derive an \emph{a posteriori} error analysis for
the $hp$-FEM~\eqref{eq:hpFEM} with respect to the residual
\[
\R:=\sup_{\genfrac{}{}{0pt}{}{v\in H^1_0(\Omega)}{v\not\equiv 0}}\frac{|a(u-\u,v)|}{\NNN{v}},
\]
where~$u\in H^1_0(\Omega)$ and~$\u\in\V$ are the exact and numerical
solutions of~\eqref{eq:var} and~\eqref{eq:hpFEM}, respectively, and~$\err=u-\u$ signifies
the error. Again, let us notice that, if~$d\ge 0$, then the
residual~$\R$ equals the norm~$\NNN{\err}$ of the error.

In order to state our main result, let us denote by~$\P$,
for~$j=1,2,\ldots,N$, the elementwise $L^2$-projection
onto~$\mathbb{P}_{p_j}(K_j)$.  Moreover, let
\[
\jmp{\u'}(x_j)=\u'(x_j^+)-\u'(x_j^-) =\lim_{x\searrow x_j}
u'(x)-\lim_{x\nearrow x_j} u'(x),\qquad 1\le j\le N-1,
\]
signify the jump of~$\u'$ at the mesh point~$x_j$, and
define~$\jmp{\u'}(x_0)=\jmp{\u'}(x_N)=0$.

\subsection{Main Result}

We shall prove the following \emph{a posteriori} error bound:

\begin{theorem}\label{thm:main}
For the error~$\err=u-\u$ between the exact solution~$u\in
H^1_0(\Omega)$ of~\eqref{eq:var} and its numerical
approximation~$\u\in\V$ from~\eqref{eq:hpFEM}, there holds the
following a posteriori error estimate:
\begin{equation}\label{eq:main}
  \R^2\le C\sum_{j=1}^N\eta_{K_j}^2.
\end{equation}
Here, for~$j=1,2,\ldots,N$,
\begin{equation}\label{eq:eta}
\begin{split}
  \eta_{K_j}^2:=\alpha_j&\left(\NN{\P
      f+\e\u''-d\u}_{L^2(K_j)}^2+\NN{f-\P f}^2_{L^2(K_j)}\right)\\
&\quad+\frac12\e^2\gamma_{j-1}\left|\jmp{\u'}(x_{j-1})\right|^2
    +\frac12\e^2\gamma_j\left|\jmp{\u'}(x_j)\right|^2
\end{split}
\end{equation}
are local error indicators, where we let
\begin{align}\label{eq:alpha}
\alpha_j&=
\begin{cases}
\min\left\{\e^{-1}h_j^2p_j^{-2},\|\nicefrac{1}{d}\|_{L^\infty(\widetilde K_j)}\right\},&\text{if }\nicefrac{1}{d}\in L^\infty(\widetilde K_j),\\
\e^{-1}h_j^2p_j^{-2},&\text{otherwise},
\end{cases}%, 
%&\widetilde\alpha_j = \max\{\alpha_{j-1},\alpha_j,\alpha_{j+1}\}
\end{align}
(with obvious modifications if $j = 0$ or $j = N$),
and
\begin{align}
\beta_j&=\alpha_jh_j^{-1}+2\sqrt{\e^{-1}\alpha_j}.\label{eq:beta}
\end{align}
Moreover, 
\begin{equation}\label{eq:gamma}
\gamma_j=\frac{\beta_j\beta_{j+1}}{\beta_j+\beta_{j+1}},
\end{equation}
for~$1\le j\le N-1$, and~$\gamma_0=\gamma_N=0$. The constant~$C>0$ is
independent of~$u$, $\u$, $f$, $\e$, $\T$, and of~$\p$.
\end{theorem}

\begin{remark}\label{rm:constants}
We emphasize that the constants~$\alpha_j$ (provided that~$\|\nicefrac{1}{d}\|_{L^\infty(\widetilde K_j)}<\infty$) and~$\varepsilon^2\alpha_j\gamma_j$ appearing in the error indicators~$\eta_{K_j}$ from~\eqref{eq:eta} remain bounded as~$h_j,\varepsilon\to 0$ (and~$p_j\to\infty$). We also note that 
$\frac{1}{2} \min\{\beta_j,\beta_{j+1}\} \leq \gamma_j \leq \min\{\beta_j,\beta_{j+1}\}$ and that 
$2\sqrt{\varepsilon^{-1} \alpha_j } \leq \beta_j \leq 3 \sqrt{\varepsilon^{-1} \alpha_j}$.
\end{remark}

\subsection{$hp$-Interpolation}\label{sc:int}

For the proof of the above Theorem~\ref{thm:main} the construction of a suitable $hp$-interpolation operator is crucial. In particular, in order to derive an (upper) \emph{a posteriori} error estimate on the error~$\err$ that is robust with respect to the singular perturbation parameter~$\e$ as well as optimally scaled with respect to the local element sizes~$h_j$ and polynomial degrees~$p_j$, an interpolant that is \emph{simultaneously} $L^2$- and $H^1$-stable is required. This will be accomplished in the current section (Proposition~\ref{pr:intp} and Corollary~\ref{cr:int}).

\begin{proposition}\label{pr:intp} 
Let the pair $(\T,\p)$ be $\mu$-shape regular 
(see \eqref{eq:gamma-shape-regular}) and~$v\in H^1_0(\Omega)$. Then, there exists an interpolant~$\intp
v\in\V$ of~$v$ such that, for any~$j=1,2,\ldots,N$, there holds
\begin{equation}\label{eq:CI}
\begin{split}
\NN{v-\intp v}_{L^2(K_j)}&\le C_I\NN{v}_{L^2(\widetilde K_j)},\qquad
\NN{v-\intp v}_{L^2(K_j)}\le C_I\frac{h_j}{p_j}\NN{v'}_{L^2(\widetilde K_j)},\\
\NN{(v-\intp v)'}_{L^2(K_j)}&\le C_I\NN{v'}_{L^2(\widetilde K_j)}.
\end{split}
\end{equation}
Furthermore, we have the nodal estimates 
\begin{equation*}
\begin{split}
|(v - \intp v)(x_i)|^2 & \leq 
C_I \Bigl[ \frac{1} {h_i + h_{i+1}} \|v - \intp v\|_{L^2(K_i \cup K_{i+1})}^2 \\ 
& \qquad+  
  \|v - \intp v\|_{L^2(K_i \cup K_{i+1})}^2 
  \|(v - \intp v)^\prime\|_{L^2(K_i \cup K_{i+1})}^2 \Bigr]. 
\end{split}
\end{equation*}
Here, $C_I>0$ is a constant that depends solely on $\mu$; in particular, 
it is independent of~$v$, $\mathcal{T}$, and of~$\bm p$. 
\end{proposition}

\begin{proof} 
Let us, without loss of generality, assume that~$\Omega=(0,1)$. The result can be shown with the techniques developed for the higher-dimensional
case in \cite{karkulik-melenk12,karkulik-melenk-rieder14}. In the present, one-dimensional case, 
a simpler argument can be brought to bear. Let $x_{-1} = -h_1$ and 
$x_{N+1} = 1+h_N$ and $\varphi_i$, $i=0,\ldots,N+1$ be the 
standard piecewise linear hat functions associated with the nodes 
$x_i$, $i=-1,\ldots,N+1$. The extra nodes $x_{-1}$ and $x_{N+1}$ define 
in a natural way the elements $K_{0}$ and $K_{N+1}$. 
The (open) patches $\omega_i$, $i=0,\ldots,N$, are given by the supports of the 
functions $\varphi_i$, i.e., $\omega_i = (\operatorname*{supp} \varphi_i)^\circ
 = K_i \cup K_{i+1} \cup \{x_i\}$. 

Polynomial approximation (see, e.g., \cite[Proposition~{A.2}]{MelenkCLEM}) gives
the existence of a interpolation operator 
$J_p:L^2(-1,1) \rightarrow {\mathbb P}_p(-1,1)$ that is uniformly 
(in $p\ge 0$)
stable, i.e., $\|J_p v\|_{L^2(-1,1)} \leq C \|v\|_{L^2(-1,1)}$ for all 
$v \in L^2(-1,1)$ and has the following properties for $v \in H^1(-1,1)$: 
\[
(p+1) \|v - J_p v\|_{L^2(-1,1)} + 
\|(v - J_p v)^\prime\|_{L^2(-1,1)} \leq  C \|v^\prime\|_{L^2(-1,1)}.  
\]
Furthermore, if $v$ is antisymmetric with respect to the midpoint $x = 0$, 
then $J_p v$ can be assumed to be antisymmetric as well, i.e., 
$(J_p v)(0) = 0$ (this follows from studying the antisymmetric part 
of the original function $J_p v$). 

The approximation $\intp v$ is now constructed with 
the aid of a ``partition of unity argument'' as described in 
\cite[Theorem~{2.1}]{babuska-melenk96}. For $\omega_0$ and $\omega_N$, 
extend $v$ anti-symmetrically, i.e., $v(x):=-v(-x)$ for $x \in K_{0}$
and $v(x):=-v(1-x)$ for $x \in K_{N+1}$. Then $v$ is defined on each
patch $\omega_i$, $i=0,\ldots,N$. For each patch $\omega_i$, 
let $p^\prime_i:= \min\{p_{i},p_{i+1}\}$ (with the understanding 
$p_0 = p_1$ and $p_{N+1} = p_N$). The above operator $J_p$ then induces
for each patch $\omega_i$ by scaling an operator 
$J^i:L^2(\omega_i) \rightarrow {\mathcal P}_{p^\prime_i-1}(\omega_i)$ 
with the following properties:  
\[
\frac{p_i^\prime+1}{h_i}\|v - J^i v\|_{L^2(\omega_i)} + 
\|(v - J^i v)^\prime\|_{L^2(\omega_i)} \leq  C \|v^\prime\|_{L^2(\omega_i)};  
\]
here, we have exploited the $\mu$-shape regularity of the mesh. 
We note that $(J^0 v)(0) = 0$ and $(J^N v)(1) = 0$. Also, the operators 
$J^i$ are uniformly (in the polynomial degree) stable in $L^2(\omega_i)$. 
The approximation $\intp v$ is now taken to be 
$\intp v := \sum_{i=0}^N \varphi_i J^i v$. The desired approximation 
properties follow now from \cite[Theorem~{2.1}]{babuska-melenk96}. 

Finally, the nodal estimate results from the observation 
that at the mesh nodes, there holds the identity $\intp v(x_i) = (J^i v)(x_i)$, and from
a multiplicative trace inequality (see Appendix~\ref{sc:AA}, Lemma~\ref{pr:trace}). 
\end{proof}

The above proposition implies the following bounds.

\begin{corollary}\label{cr:int}
For~$v\in H^1_0(\Omega)$, the interpolant from Proposition~\ref{pr:intp} satisfies
\begin{equation*}
\NN{v-\intp v}^2_{L^2(K_j)}\le
C^2_I\alpha_j\NNN{v}_{\widetilde K_j}^2 ,
\qquad j=1,2,\ldots,N,
\end{equation*}
and
\begin{equation*}
\big|(v-\intp v)(x_j)\big|^2\le
C_I^2\gamma_j\left(\NNN{v}_{\widetilde K_{j}}^2+\NNN{v}^2_{\widetilde K_{j+1}}\right),
\qquad j=1,2,\ldots,N-1,
\end{equation*}
where~$\alpha_j$ and~$\gamma_j$ are defined
in~\eqref{eq:alpha} and~\eqref{eq:gamma}, respectively, and~$C_I$ is the constant
from~\eqref{eq:CI}.
\end{corollary}

\begin{proof}
We proceed along the lines of~\cite{Ve1998}. Using the bounds from
Proposition~\ref{pr:intp}, we have for each element~$K_j\in\T$ that
\begin{align*}
\NN{v-\intp v}^2_{L^2(K_j)}&\le 
C^2_I\frac{h_j^2}{\e p_j^2}\e\NN{v'}^2_{L^2(\widetilde K_j)}.
\end{align*}
Furthermore, if~$\nicefrac{1}{d} \in L^\infty(\widetilde K_j)$, then
\begin{align*}
\NN{v-\intp v}^2_{L^2(K_j)}
&\le C^2_I\NN{v}_{L^2(\widetilde K_j)}^2\le C^2_I
\NN{\nicefrac{1}{d}}_{L^\infty(\widetilde K_j)}\NN{\sqrt{|d|}v}_{L^2(\widetilde K_j)}^2.
\end{align*}
Combining these two estimates, yields the first bound.
 
In order to prove the second estimate, we apply, for~$1\le j\le N-1$,
a multiplicative trace inequality (see Appendix~\ref{sc:AA},
Lemma~\ref{pr:trace}):
\begin{align*}
\big|(v&-\intp v)(x_j)\big|^2\\
&\le h_j^{-1}\NN{v-\intp v}^2_{L^2(K_j)}+2\NN{v-\intp
  v}_{L^2(K_j)}\NN{(v-\intp v)'}_{L^2(K_j)}.
\end{align*}
Then, invoking the above bounds as well as the estimates from
Proposition~\ref{pr:intp}, we get 
\begin{align*}
\big|(v-\intp v)(x_j)\big|^2
&\le C_I^2\left(\alpha_jh_j^{-1}\NNN{v}_{\widetilde K_j}^2
+2\sqrt{\alpha_j}\NNN{v}_{\widetilde K_j}\NN{v'}_{L^2(\widetilde K_j)}\right)\\
&\le C_I^2\left(\alpha_jh_j^{-1}\NNN{v}_{\widetilde K_j}^2
+2\sqrt{\e^{-1}\alpha_j}\NNN{v}_{\widetilde K_j}^2\right)\\
&\le C_I^2\beta_j\NNN{v}_{\widetilde K_j}^2,
\end{align*}
with $\beta_j$ from~\eqref{eq:beta}. Since~$x_j$ is also a boundary
point of~$K_{j+1}$, we similarly obtain that
\[
\big|(v-\intp v)(x_j)\big|^2
\le C_I^2\beta_{j+1}\NNN{v}_{\widetilde K_{j+1}}^2.
\]
Therefore,
\begin{align*}
\big|(v-\intp v)(x_j)\big|^2
&=\frac{\beta_{j+1}}{\beta_{j}+\beta_{j+1}}\big|(v-\intp
v)(x_j)\big|^2 +\frac{\beta_{j}}{\beta_{j}+\beta_{j+1}}\big|(v-\intp
v)(x_j)\big|^2\\ &\le
C_I^2\gamma_j\left(\NNN{v}_{\widetilde K_{j}}^2+\NNN{v}^2_{\widetilde K_{j+1}}\right),
\end{align*}
with~$\gamma_j$ from~\eqref{eq:gamma}. Thus, we have shown the second
estimate.
\end{proof}

\subsection{Proof of Theorem~\ref{thm:main}}

We are now in a position to prove the $hp$-\emph{a posteriori} error
bound~\eqref{eq:main}. 

From the definitions of the exact solution~$u$ from~\eqref{eq:var} and
the numerical solution~$\u$ defined in~\eqref{eq:hpFEM}, it follows
that, for any~$v\in H^1_0(\Omega)$ and any~$\vhp\in\V$,
\begin{align*}
a(u,v)-a(\u,v)&=a(u,v-\vhp)-a(\u,v-\vhp)\\
&=\int_\Omega f(v-\vhp)\dx-\e\int_\Omega \u'(v-\vhp)'\dx-\int_\Omega d\u(v-\vhp)\dx.
\end{align*}
Integrating by parts elementwise in the second integral leads to
\begin{align*}
\int_\Omega &\u'(v-\vhp)'\dx
=\sum_{j=1}^N\int_{K_j}\u'(v-\vhp)'\dx\\
&=-\sum_{j=1}^N\int_{K_j}\u''(v-\vhp)\dx
+\sum_{j=1}^N\left(\u'(x^-_j)(v-\vhp)(x_j)
-\u'(x^+_{j-1})(v-\vhp)(x_{j-1})\right)\\
&=-\sum_{j=1}^N\int_{K_j}\u''(v-\vhp)\dx-\sum_{j=1}^{N-1}\jmp{\u'}(x_j)(v-\vhp)(x_j),
\end{align*}
and thus, choosing $\vhp=\intp v$ to be the $hp$-interpolant from Section~\ref{sc:int}, we arrive at
\begin{align*}
a(u,v)-a(\u,v)
&=\sum_{j=1}^N\left(\P f+\e\u''-d\u\right)(v-\intp v)\dx\\
&\quad+\sum_{j=1}^N\left(f-\P f\right)(v-\intp v)\dx
+\e\sum_{j=1}^{N-1}\jmp{\u'}(x_j)(v-\intp v)(x_j).
\end{align*}
Hence, applying the Cauchy-Schwarz inequality, we obtain
\begin{align*}
|a(u,v)-a(\u,v)|
&\le
\sum_{j=1}^N\NN{\P f+\e\u''-d\u}_{L^2(K_j)}\NN{v-\intp v}_{L^2(K_j)}\\
&\quad+\sum_{j=1}^N\NN{f-\P f}_{L^2(K_j)}\NN{v-\intp v}_{L^2(K_j)}\\
&\quad+\sum_{j=1}^{N-1}\e\N{\jmp{\u'}(x_j)}\N{(v-\intp v)(x_j)}.
\end{align*}
The bounds from Corollary~\ref{cr:int} lead to
\begin{align*}
  |a(u,v)-a(\u,v)| &\le
  C_I\sum_{j=1}^N\sqrt{\alpha_j}\NN{\P f+\e\u''-d\u}_{L^2(K_j)}\NNN{v}_{\widetilde K_j}\\
  &\quad+C_I\sum_{j=1}^N\sqrt{\alpha_j}\NN{f-\P f}_{L^2(K_j)}\NNN{v}_{\widetilde K_j}\\
  &\quad+C_I\sum_{j=1}^{N-1}\left(\NNN{v}_{\widetilde K_{j}}^2+\NNN{v}^2_{\widetilde K_{j+1}}\right)^{\nicefrac12}\e\sqrt{\gamma_j}\N{\jmp{\u'}(x_j)}.
\end{align*}
The Cauchy-Schwarz inequality yields
\begin{align*}
  |a(u,v)&-a(\u,v)|\\
&\le
  C_I\left(\sum_{j=1}^N\alpha_j\NN{\P f+\e\u''-d\u}^2_{L^2(K_j)}
+\alpha_j\NN{f-\P f}_{L^2(K_j)}^2\right)^{\nicefrac12}\left(2\sum_{j=1}^N\NNN{v}^2_{\widetilde K_j}\right)^{\nicefrac12}\\
  &\quad+C_I\left(\sum_{j=1}^{N-1}\e^2\gamma_j\N{\jmp{\u'}(x_j)}^2\right)^{\nicefrac12}\left(\sum_{j=1}^{N-1}\left(\NNN{v}_{\widetilde K_{j}}^2+\NNN{v}^2_{\widetilde K_{j+1}}\right)\right)^{\nicefrac12}
\end{align*}
Observing that
\[
\sum_{j=1}^N\NNN{v}^2_{\widetilde K_j}\le 3\NNN{v}^2,\qquad
\sum_{j=1}^{N-1}\left(\NNN{v}_{\widetilde K_{j}}^2+\NNN{v}^2_{\widetilde K_{j+1}}\right)\le 6\NNN{v}^2,
\]
we finally see that
\begin{align*}
|a(u,v)&-a(\u,v)|\
\le \sqrt{12}C_I\left(\sum_{j=1}^N\eta_{K_j}^2\right)^{\nicefrac12}\NNN{v},
\end{align*}
with~$\eta_{K_j}$ from~\eqref{eq:eta}. Dividing both sides of this
inequality by~$\NNN{v}$ and taking the supremum for all~$v\in
H^1_0(\Omega)$ shows Theorem~\ref{thm:main}.

\begin{remark} In the case~$d > 0$, following along the lines of~\cite{Ve1998} and \cite{MelenkAPOST01}, and employing $p$-dependent norm equivalence estimates in order to be able to involve suitable cut-off functions locally, it is possible to prove $\e$-robust local lower bounds for the error in terms of the error indicators~$\eta_{K_j}$ and some data oscillation terms. Specifically, if $d$ satisfies 
$0 < d_0 \leq \inf_{x \in \Omega} d(x) \leq \sup_{x \in \Omega} d(x) \leq d_1 < \infty$ and $\beta \in (\nicefrac12,1]$
is fixed, then, the lower bounds 
\begin{align*}
\alpha_{j} \|f - (-\e \u^{\prime\prime} + d \u)\|^2_{L^2(K_j)}
& \leq C \left[ p_j^2 \NNN{u - \u}^2_{K_j}  + \alpha_j R_{K_j}^2\right],\quad 1\le j\le N,
\end{align*}
and
\begin{align*}
\gamma_j \e^2 |\jmp{\u}(x_j)|^2 &\leq C \left[ p_{j}^2 \NNN{u - \u}^2_{K_j \cup K_{j+1}} + \alpha_j R_{K_j}^2 + 
\alpha_{i+1} R_{K_{j+1}}^2\right],\quad 1\le j\le N-1,
\end{align*}
can be proved.
Here, for any element~$K_j\in\T$, $1\le j\le N$, the data oscillation term $R_{K_j}$ is defined by 
\begin{align*}
R_{K_j} &= 
{p_j^\beta}
\left[ 
\left\|\Phi_{K_j}^{\nicefrac{\beta}{2}} (f - \Pi_{K_j} f)\right\|_{L^2(K_j)} + 
\left\|\Phi_{K_j}^{\nicefrac{\beta}{2}} (d \u - \Pi_{K_j} (d \u))\right\|_{L^2(K_j)} 
\right] \\
&\quad+ 
\|f - \Pi_{K_j} f\|_{L^2(K_j)} + 
\|d \u - \Pi_{K_j} (d \u)\|_{L^2(K_j)}.  
\end{align*}
The constant $C > 0$ depends only on the ratio $\nicefrac{d_1}{d_0}$, the choice of $\beta \in (\nicefrac12,1]$, and the shape-regularity
parameter $\mu$ from~\eqref{eq:gamma-shape-regular}; see Appendix~\ref{sc:AB} (in particular, Theorem~\ref{thm:efficiency-appendix}) for details. It is worth stressing that the $L^2$-projector $\Pi_{K_j}:L^2(K_j) \rightarrow \mathbb{P}_{p_j}(K_j)$ can be replaced with a projection onto a space of polynomials of degree $\lambda p_j$ for a fixed $\lambda > 0$. While the constant $C$ then additionally depends on $\lambda$, this allows to exploit smoothness of the coefficient function $d$  in the treatment of the second term in~$R_{K_j}$.
\end{remark}

%%%%%%%%%%%%%%%%%%%%%%%%%%%%%%%%%%%%%%%%%%%%%%%%%%%%%%%%%%%%%%%%%%%

\section{Numerical Experiments}\label{sc:numerics}

The purpose of this section is to illustrate the \emph{a posteriori} error estimates from Theorem~\ref{thm:main} in the context of some specific numerical experiments. We will emphasize on the robustness of the error indicators with respect to~$\e$ as~$\e\to 0$, and on the capability of $hp$-FEM to deliver exponential rates of convergence. 

\subsection{$hp$-Adaptive Procedure}
We shall apply an $hp$-adaptive algorithm which is based on the following ingredients:
\begin{enumerate}[(a)]
\item \emph{Element marking:} The elementwise error indicators~$\eta_{K_j}$ from~Theorem~\ref{thm:main} are employed in order to mark elements for refinement. More precisely, we fix a parameter~$\theta\in(0,1)$ (in the experiments below we choose~$\theta=0.5)$ and select elements to be refined according to the \emph{D\"orfler marking} criterion:
\begin{equation}\label{eq:dorf}\tag{D}
\theta\sum_{j=1}^N\eta_{K_j}^2\le \sum_{j'=1}^M\eta_{K_{j'}}^2.
\end{equation}
Here, the indices~$j'$ are chosen such that the error indicators~$\eta_{K_{j'}}$ from~\eqref{eq:eta} are sorted in descending order, and~$M$ is minimal.
\item \emph{$hp$-refinement criterion:} The decision of whether a marked element in step~(a) is refined with respect to~$h$ (element bisection) or~$p$ (increasing the local polynomial order by~1) is based on a smoothness testing approach. Specifically, if the (numerical) solution is considered smooth on a marked element~$K_j$, then the polynomial degree is increased by~1 on that particular element (no element bisection), otherwise the element is bisected (retaining the current polynomial degree~$p_j$ on both subelements). In order to evaluate the smoothness of the solution~$\u$ on a marked element~$K_j$, we employ an elementwise smoothness indicator as introduced in~\cite[Eq.~(3)]{FaWiWi14}:
\begin{equation}\label{eq:hp}\tag{F}
\mathcal{F}^{p_j}_j[\u]:=\begin{cases}\displaystyle
\frac{\NN{\frac{\dd^{p_j-1}}{\dd x^{p_j-1}}\u}_{L^\infty(K_j)}}{h_j^{-\nicefrac12}\NN{\frac{\dd^{p_j-1} \u}{\dd x^{p_j-1}}}_{L^2(K_j)}+\frac{1}{\sqrt2}h_j^{\nicefrac12}\NN{\frac{\dd^{p_j} \u}{\dd x^{p_j}}}_{L^2(K_j)}}  & \text{if }\frac{\dd^{p_j-1}}{\dd x^{p-1}}\u|_{K_j}\not\equiv 0,\\[3ex]
1 & \text{if }\frac{\dd^{p_j-1}}{\dd x^{p_j-1}}\u|_{K_j}\equiv 0.
\end{cases}
\end{equation}
Here, the basic idea is to consider the continuous Sobolev embedding~$H^1(K_j)\hookrightarrow L^\infty(K_j)$, which implies that
\[
\sup_{v\in H^1(K_j)}\frac{\NN{v}_{L^\infty(K_j)}}{h_j^{-\nicefrac12}\NN{v}_{L^2(K_j)}+\frac{1}{\sqrt2}h_j^{\nicefrac12}\NN{v'}_{L^2(K_j)}}\le 1;
\]
see~\cite[Proposition~1]{FaWiWi14}. In particular, it follows that~$\mathcal{F}_j^{p_j}[\u]\le 1$. For ease of evaluation, note that, by taking the derivative of order~$p_j-1$ in the definition~\eqref{eq:hp}, the smoothness indicator~$\mathcal{F}^{p_j}_j[\u]$ is evaluated for linear functions only; in this case, it can be shown that
\[
\frac12\approx\frac{\sqrt{3}}{\sqrt{6}+1}\le\mathcal{F}_j^{p_j}[\u]\le 1;
\] 
cf.~\cite[Section~2.2]{FaWiWi14}. The numerical solution~$\u$ is classified smooth on~$K_j$ if~$\mathcal{F}^{p_j}_j[\u]\ge\tau$ and otherwise nonsmooth, for a prescribed smoothness testing parameter~$\tau\in(\nicefrac{\sqrt{3}}{(\sqrt{6}+1)},1)$ (in our experiments we choose~$\tau=0.6$). Incidentally, representing the local solution~$\u|_{K_j}$ in terms of (local) Legendre polynomials (or more general Jacobi polynomials), any derivatives of~$\u$ can be evaluated exactly by means of appropriate recurrence relations. We refer to the papers~\cite{FaWiWi14} (see also~\cite{Wi11_2,Wi11}) for more details on this smoothness testing strategy.
\end{enumerate}

Combing the above ideas leads to the following $hp$-adaptive refinement algorithm:

\begin{algorithm}\label{al:hp}
Choose prescribed parameters~$\theta\in(0,1)$ and~$\tau\in \left(\frac{\sqrt{3}}{\sqrt{6}+1},1\right)$ for the D\"orfler marking as well as for the $hp$-decision process as described before, respectively.
Furthermore, consider a (coarse) initial mesh~$\T^0$, and an associated polynomial degree vector~$\bm p^0$. Set~$n=0$. Then, perform the following iteration (until a given maximum iteration number is reached, or until the estimated error is sufficiently small):
\begin{enumerate}[(1)]
\item Compute the numerical solution~$\u^n\in\Vn$ from~\eqref{eq:hpFEM}, and evaluate the error indicators~$\{\eta_{K_j}\}_{K_j\in\T^n}$ defined in~\eqref{eq:eta}.
\item Mark the elements in~$\mathcal{T}^n$ based on the D\"orfler marking~\eqref{eq:dorf}.
\item Create the mesh $\mathcal{T}^{n+1}$ with corresponding polynomial degree distribution $\bm p^{n+1}$: 
For each marked element~$K_j \in \mathcal{T}^n$ evaluate the smoothness indicator~$\mathcal{F}^{p_j}_j[\u]$ from~\eqref{eq:hp}; if there holds~$\mathcal{F}^{p_j}_j[\u]\ge\tau$ then increase the polynomial degree~$p_j^n$ by 1, i.e., $p_j^n\leftarrow p_j^n+1$, otherwise bisect~$K_j$ into two new elements (taking~$p_j^n$ for both elements). Increase $n$ by $1$, i.e., $n\leftarrow n+1$.
\end{enumerate}
\end{algorithm}

In the ensuing experiments, we will start Algorithm~\ref{al:hp} based on a uniform initial mesh consisting of 10 elements, and a polynomial degree distribution~$\bm p^0=(1,\ldots 1)$.

\subsection{Example 1:} We begin by looking at the singularly perturbed reaction-diffusion problem
\[
-\e u''+u=1\quad\text{on }\Omega=(-1,1),\qquad u(-1)=u(1)=0.
\]
This problem is coercive and has exactly one (analytic) solution. For small~$\e\ll 1$ the exact solution exhibits a boundary layer at~$x=0$ and~$x=1$ which needs to be resolved properly by the $hp$-adaptive FEM. In Figure~\ref{fig:Ex1a} the $hp$-mesh after 24 adaptive refinement steps is displayed for~$\e=10^{-4}$. We observe that the boundary layer is resolved by some mild $h$-refinement and by increasing~$p$ in the same area. Moreover, the mesh remains unrefined in the center of the domain where the exact solution is nearly constant~$1$. In addition, in Figures~\ref{fig:Ex1b_1} and~\ref{fig:Ex1b_2} we show the errors measured with respect to the norm~$\NNN{\cdot}$ from~\eqref{eq:norm} as well as the estimated errors. The exponential decay of both quantities for different choices of~$\e$ becomes clearly visible in the semi-logarithmic plot. Finally, the efficiency indices, i.e., the ratio between the estimated and true errors, are depicted in Figure~\ref{fig:Ex1c}; they oscillate between~$1$ and~$4$, and do not deteriorate as~$\e\to 0$, thereby clearly testifying to the robustness of the \emph{a posteriori} error estimate from Theorem~\ref{thm:main}. 

\begin{figure}
\begin{center}
	\includegraphics[width=0.8\linewidth]{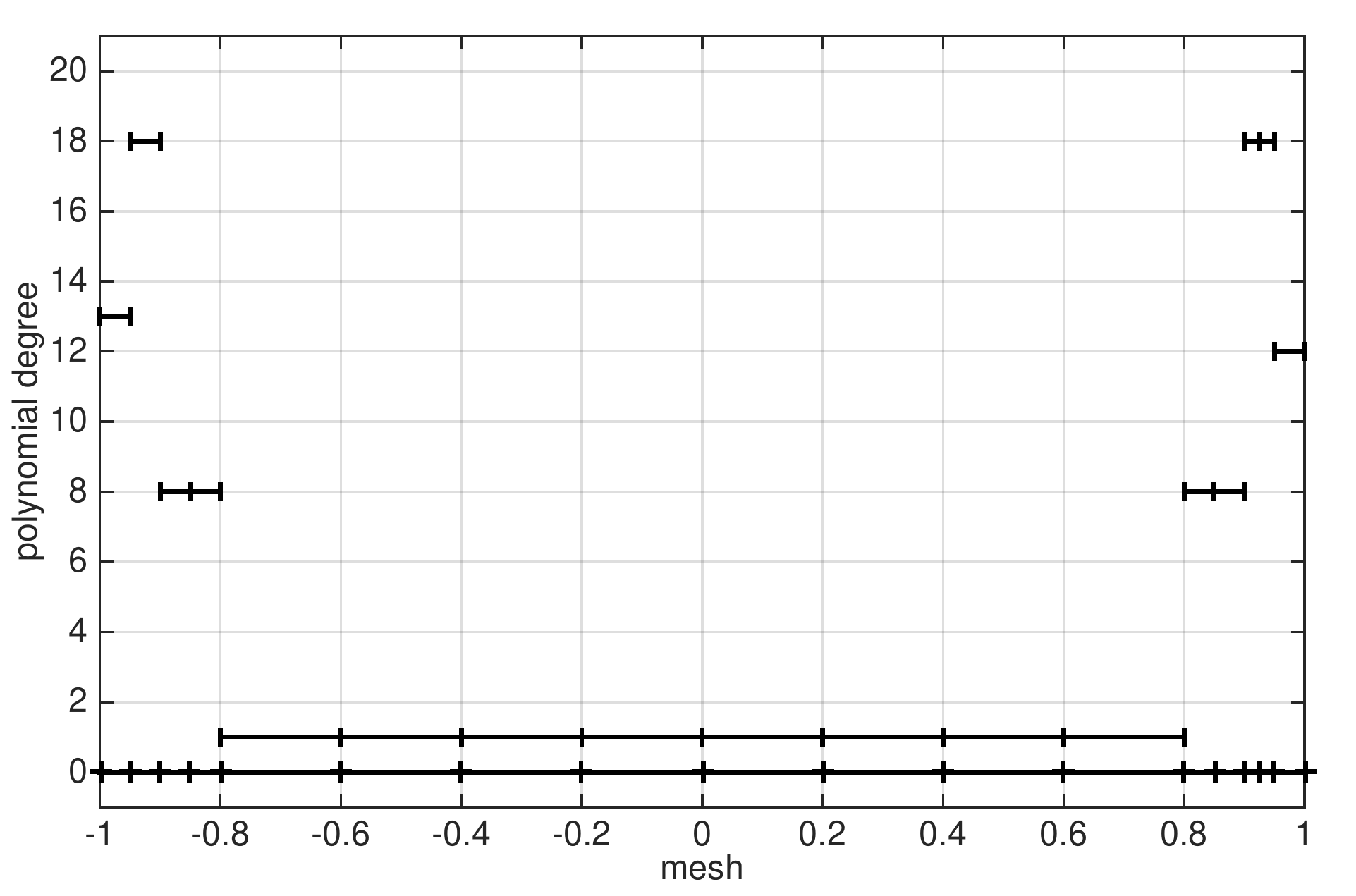}
\end{center}
\caption{Example 1 for~$\e=10^{-4}$: Adaptively generated $hp$-mesh after 24 refinement steps (17 elements, maximal polynomial degree 18).}
\label{fig:Ex1a}
\end{figure}

\begin{figure}
\begin{center}
	\includegraphics[width=0.8\linewidth]{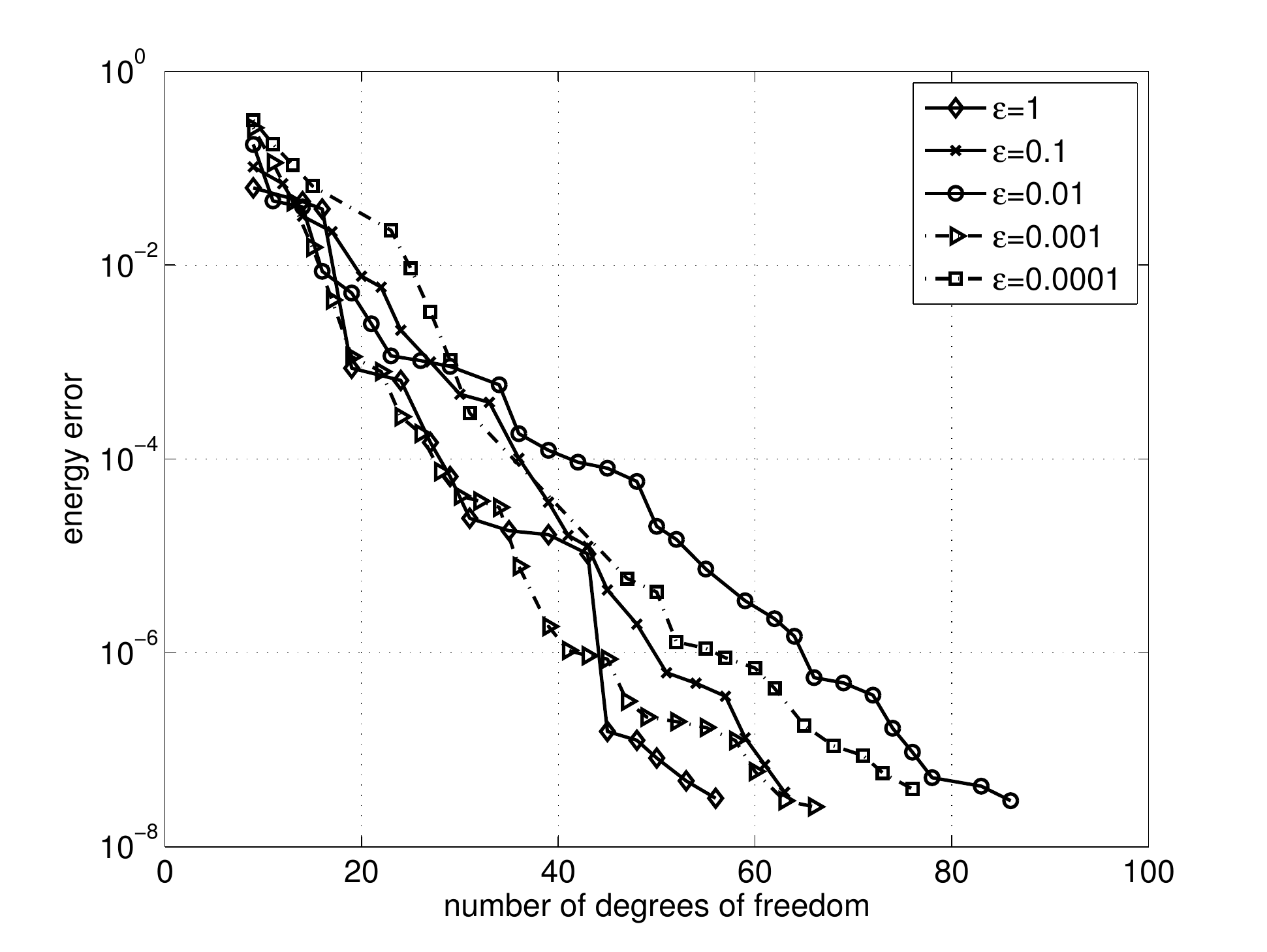}\\
\end{center}
\caption{Example 1: Energy error for different choices of~$\e$.}
\label{fig:Ex1b_1}
\end{figure}

\begin{figure}
\begin{center}
	\includegraphics[width=0.8\linewidth]{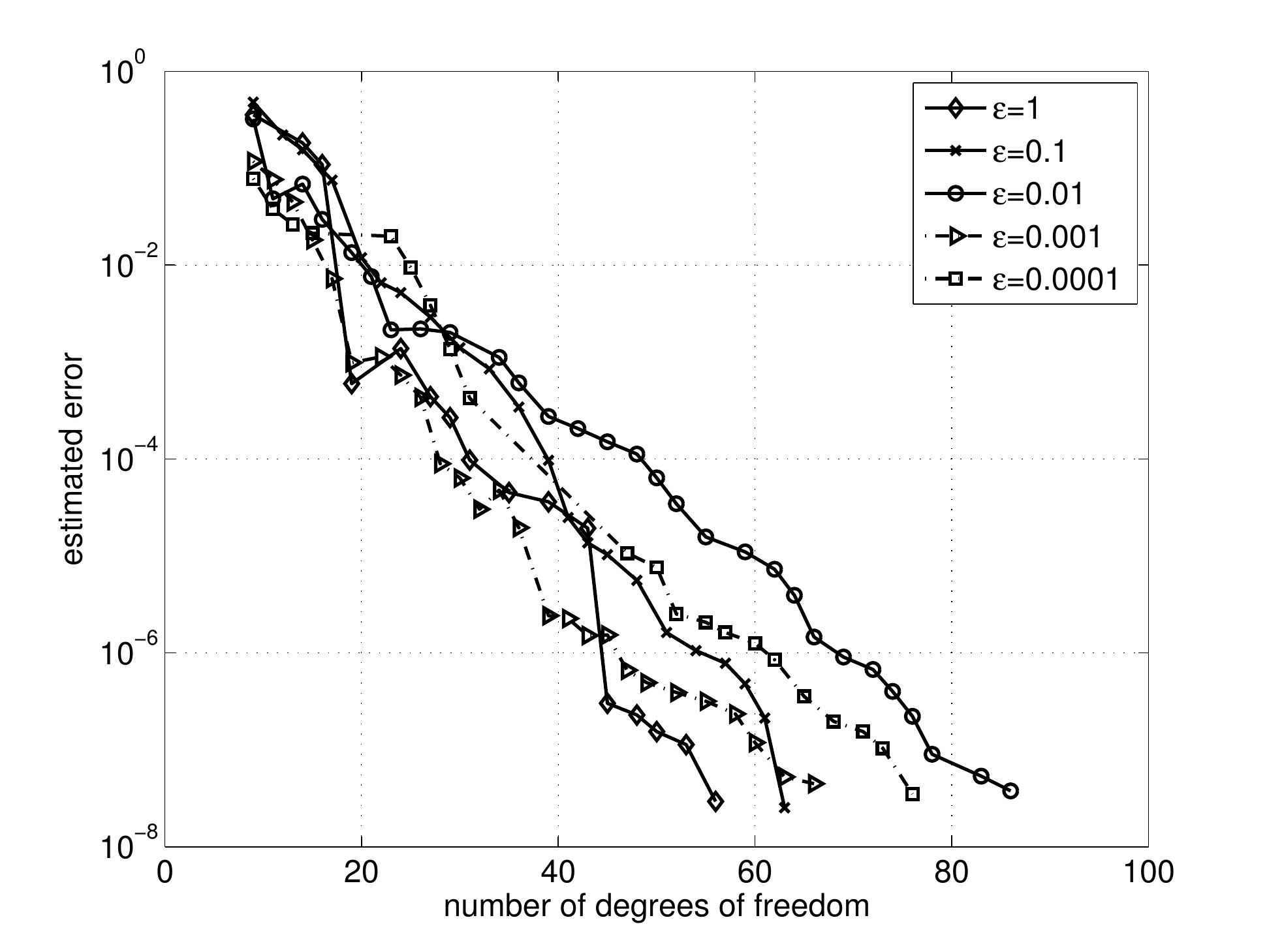}
\end{center}
\caption{Example 1: Estimated error for different choices of~$\e$.}
\label{fig:Ex1b_2}
\end{figure}

\begin{figure}
\begin{center}
	\includegraphics[width=0.8\linewidth]{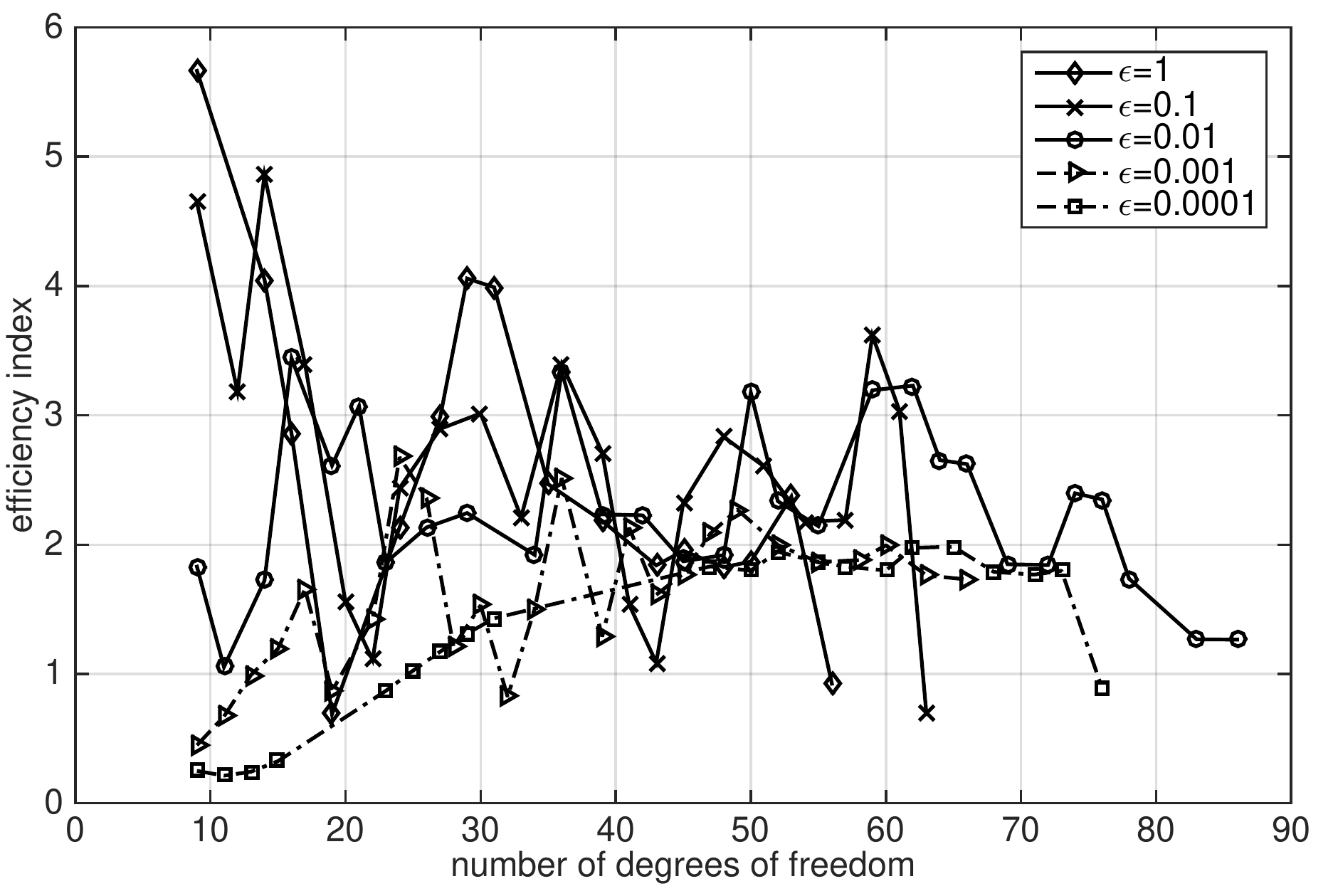}\hfill
\end{center}
\caption{Example 1: Efficiency indices for different choices of~$\e$.}
\label{fig:Ex1c}
\end{figure}

\subsection{Example 2:} In this experiment, we consider Airy's equation
\[
-\e u''+xu=1\quad\text{on }\Omega=(-1,1),\qquad u(-1)=u(1)=0.
\]
The particularity of this example is that, for~$0<\e\ll 1$, the corresponding differential operator is coercive for~$x\ge 1$, however, it becomes hyperbolic near~$x=-1$; this becomes evident in Figure~\ref{fig:Ex2a_1}, where the numerical solution is shown for~$\e=10^{-4}$. The oscillating regime for~$x<0$ requires a proper resolution by the $hp$-FEM as shown in the $hp$-mesh in Figure~\ref{fig:Ex2c}. The decay of the estimated error is plotted in Figure~\ref{fig:Ex2a_2} for various choices of~$\e$. In particular, for small~$\e$, we see that, after a number of initial refinements resolving the oscillations, the algorithm provides exponentially converging results.

\begin{figure}
\begin{center}
	\includegraphics[width=0.8\linewidth]{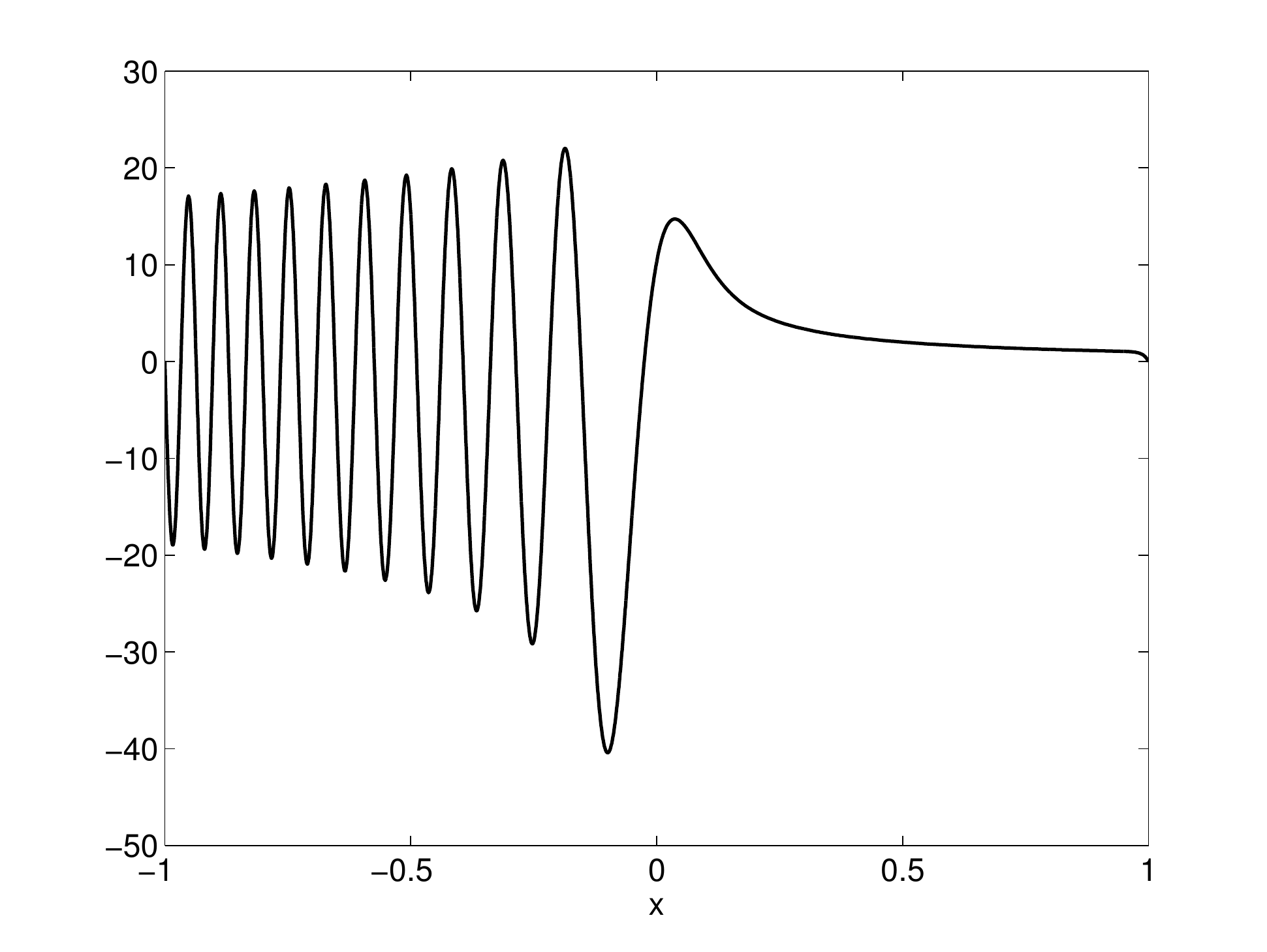}\\
\end{center}
\caption{Example 2 for~$\e=10^{-4}$: Numerical solution with strong oscillations for $x < 0$.}
\label{fig:Ex2a_1}
\end{figure}

\begin{figure}
\begin{center}
	\includegraphics[width=0.8\linewidth]{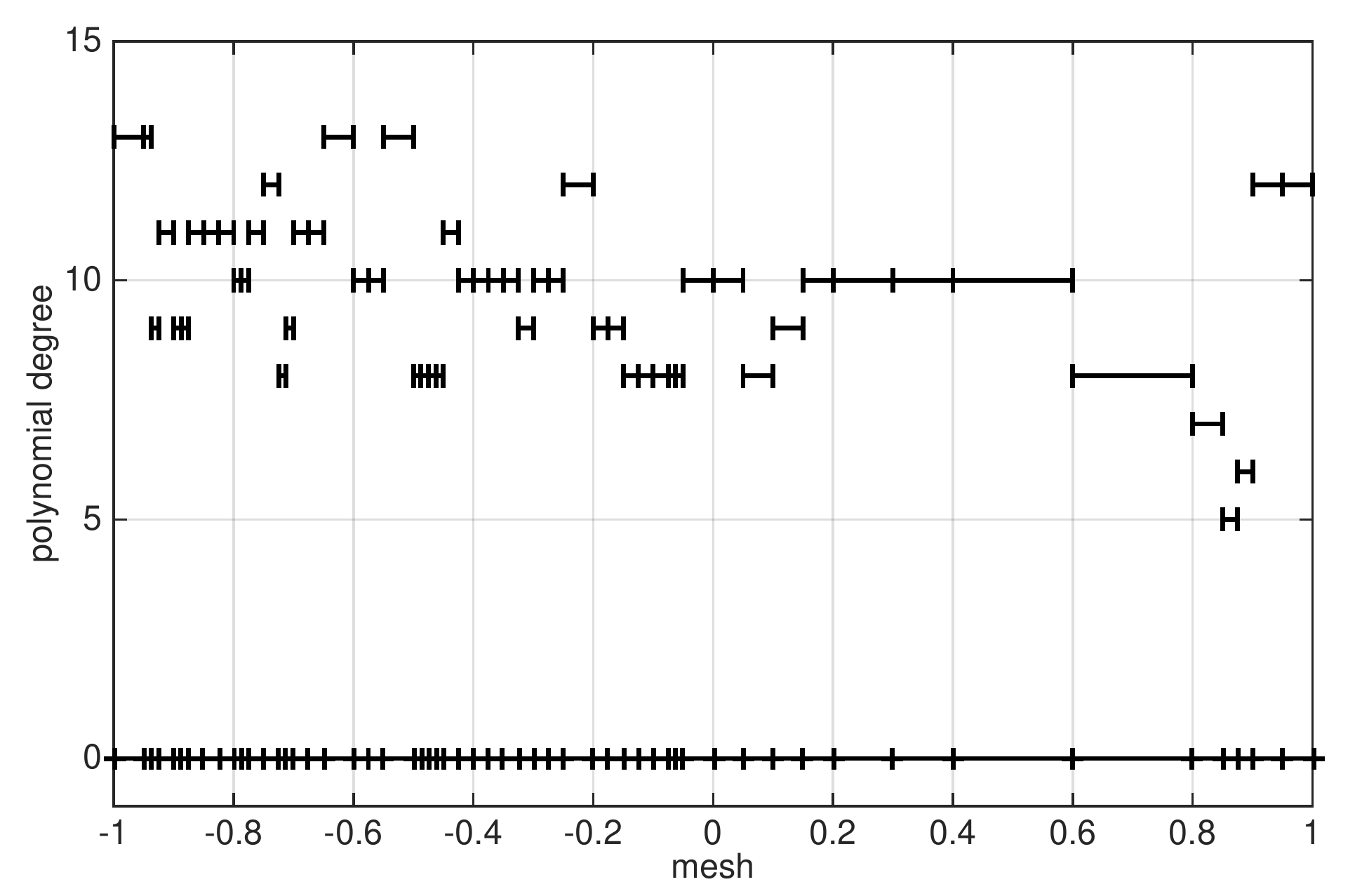}
\end{center}
\caption{Example 2 for~$\e=10^{-4}$: Adaptively generated $hp$-mesh after 75 refinement steps (55 elements, maximal polynomial degree 13).}
\label{fig:Ex2c}
\end{figure}

\begin{figure}
\begin{center}
	\includegraphics[width=0.8\linewidth]{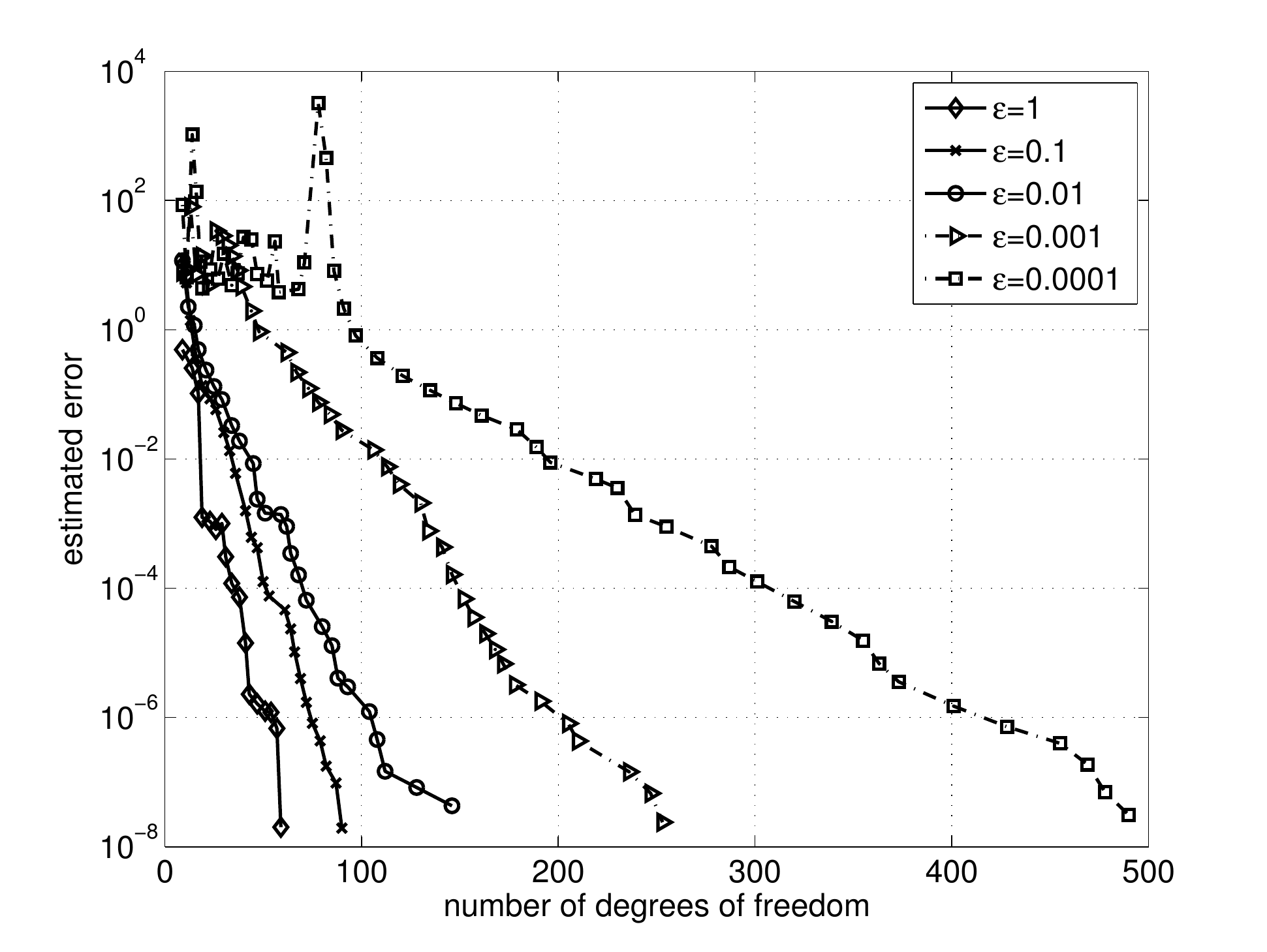}
\end{center}
\caption{Example 2: Estimated errors for different choices of~$\e$.}
\label{fig:Ex2a_2}
\end{figure}

%%%%%%%%%%%%%%%%%%%%%%%%%%%%%%%%%%%%%%%%%%%%%%%%%%%%%%%%%%%%%%%%%%%

\section{Conclusions}\label{sc:concl}

In this paper we have studied the numerical approximation of linear second-order boundary value problems (with possibly non-constant reaction coefficient) by the $hp$-FEM. In particular, we have derived an \emph{a posteriori} error estimate for a natural residual-type norm that is robust with respect to the (possibly) small perturbation parameter and explicit with respect to the local mesh size and polynomial degree. Numerical experiments for both coercive as well as partly coercive differential equations underline the robustness of the error bound. In addition, an appropriate combination of the error estimate with a smoothness testing procedure reveals that the method is able to achieve exponential rates of convergence.

%%%%%%%%%%%%%%%%%%%%%%%%%%%%%%%%%%%%%%%%%%%%%%%%%%%%%%%%%%%%%%%%%%%

\appendix

\clearpage

\section{Multiplicative Trace Inequality}\label{sc:AA}

\begin{lemma}\label{pr:trace}
  Let~$h>0$ and~$w\in H^1(0,h)$. Then, the multiplicative trace
  inequality
\[
\max\left\{|w(0)|, |w(h)|\right\}^2\le
h^{-1}\|w\|^2_{L^2(0,h)}+2\|w\|_{L^2(0,h)}\|w'\|_{L^2(0,h)}
\]
holds true.
\end{lemma}

\begin{proof}
  By density of~$C^\infty([0,h])$ in~$H^1(0,h)$, we may suppose
  that~$w$ is smooth. There holds
\[
w(0)^2=\int_0^h\frac{\mathsf{d}}{\dx}\left[\left(h^{-1}x-1\right)w(x)^2\right]\dx
=h^{-1}\int_0^h
w(x)^2\dx+2\int_0^h\left(h^{-1}x-1\right)w(x)w'(x)\dx.
\]
Then, applying the Cauchy-Schwarz inequality and noticing
that~$\left|1-h^{-1}x\right|<1$ for~$x\in(0,h)$, results in
\[
|w(0)|^2\le h^{-1}\|w\|^2_{L^2(0,h)}+2\|w\|_{L^2(0,h)}\|w'\|_{L^2(0,h)}.
\]
By symmetry, the same bound can be obtained for~$|w(h)|^2$. This completes the
proof.
\end{proof}

\section{Details on the efficiency bound}\label{sc:AB}

We follow \cite{MelenkAPOST01}, taking care of the presence of the singular perturbation
parameter $\e$ as well as the fact that the coefficient $d$ is possibly variable. 
For an element $K_i=(x_{i-1},x_{i})$, $1\le i\le N$, let $\Phi_{K_i}$ be the scaled distance function from $\partial K_i=\{x_{i-1},x_i\}$, i.e., 
\[
\Phi_{K_i}(x) = h_i^{-1} \min(|x-x_{i-1}|,|x-x_i|),\qquad x\in K_i. 
\]

\begin{lemma}
\label{lemma:MW-lemma-3.3}
Let~$\u$ be the $hp$-FEM solution of~\eqref{eq:hpFEM}, and~$K_i\in\T$, $1\le i\le N$. Then, for any $\beta \in (-\nicefrac12,1]$ there exists a constant $C_\beta > 0$ such that 
\begin{equation}\label{eq:AB1}
\begin{split}
\|f - (-\e \u^{\prime\prime} +& d \u)\|_{L^2(K_i)} \\
\leq C_\beta \Bigg\{&
\left( 
\sqrt{\e} \frac{p_i^2}{h_i} + {p_i^\beta} \left\|\sqrt{|d|}\right\|_{L^\infty(K_i)}
\right) \NNN{u - \u}_{K_i} \\
&  + 
{p_i^\beta}
\left[ 
\left\|\Phi_{K_i}^{\nicefrac{\beta}{2}} (f - \Pi_{K_i} f)\right\|_{L^2(K_i)} + 
\left\|\Phi_{K_i}^{\nicefrac{\beta}{2}} (d \u - \Pi_{K_i} (d \u))\right\|_{L^2(K_i)} 
\right] \\
&
+ 
\|f - \Pi_{K_i} f\|_{L^2(K_i)} + 
\|d \u - \Pi_{K_i} (d \u)\|_{L^2(K_i)} 
\Bigg\}.
\end{split}
\end{equation}
\end{lemma}

\begin{remark}~
\label{rem:efficiency-volume-term}
\begin{enumerate}[(i)]
\item 
As written in Lemma~\ref{lemma:MW-lemma-3.3}, $\Pi_{K_i}$ signifies the $L^2(K_i)$-projection. This is not essential and could be replaced with other approximation operators. It is also not necessary that $\Pi_{K_i}$ maps into the space of polynomials of degree $p_i$---it could as well be the space of degree~$2 p_{i}$. 
\item 
The right-hand side of~\eqref{eq:AB1} involves  the $hp$-FEM solution~$\u$ from~\eqref{eq:hpFEM}. If $\Pi_{K_i}$ maps into the space of polynomials of degree $2 p_{i}$, then the term $d \u - \Pi_{i} (d \u)$ can be controlled provided some \emph{a priori} control of $ \|\u\|_{L^2(\Omega)}$ or at least of $\|\sqrt{|d|} \u \|_{L^2(\Omega)}$ is available. 
\item For those elements where $d$ is bounded away from $0$, Lemma~\ref{lemma:MW-lemma-3.3} provides indeed a lower bound since then 
$$
\sqrt{\alpha_i} \sim 
\min\left \{\frac{h_i}{\sqrt{\e} p_i}, \left\| \frac{1}{\sqrt{|d|}}\right\|_{L^\infty(K_i)}\right\}
$$
is of order~$\mathcal{O}(1)$. In fact, if $\inf_{x \in K_i} |d(x)|$ and $\sup_{x \in K_i} |d(x)|$ are of comparable magnitude, then 
\begin{equation}
\label{eq:foo-10}
\sqrt{\alpha_i} \left( \sqrt{\e} \frac{p_i^2}{h_i} + p_i^\beta \left\|\sqrt{|d|}\right\|_{L^\infty(K_i)} \right) 
\leq C p_i,
\end{equation}
where the constant $C>0$ depends only on the ratio 
\begin{equation}
\label{eq:ratio}
\frac{\sup_{x \in K_i} |d(x)|}{\inf_{x \in K_i} |d(x)|}.
\end{equation}
\item
If the  ratio \eqref{eq:ratio} cannot be controlled well (e.g., if $|d|$ becomes arbitrarily small or even zero on~$K_i$), then the efficiency bound breaks down unless the element is sufficiently small (relative to $\e$). 
\end{enumerate}
\end{remark}

\begin{proof}[Proof of Lemma~\ref{lemma:MW-lemma-3.3}]
Let $\beta \in (0,1]$. On~$K_i$ define
\[
v_{K_i}:= \Phi_{K_i}^\beta \cdot (\Pi_{K_i} (f|_{K_i}) - (-\e (\u|_{K_i})^{\prime\prime} + \Pi_{K_i}(d \u|_{K_i}))).
\] 
We write 
\begin{align*}
\left\|\Phi_K^{-\nicefrac{\beta}{2}} v_{K_i}\right\|^2_{L^2(K_i)} &=\int_{K_i} (\Pi_{K_i} f - (-\e \u^{\prime\prime} + \Pi_{K_i} (d \u)) v_{K_i}\dx \\
&= \int_{K_i} (f - (-\e \u^{\prime\prime} + d \u)) v_{K_i} \dx + 
   \int_{K_i} (\Pi_{K_i} f - f) v_{K_i} \dx \\
&   \quad- 
   \int_{K_i} (\Pi_{K_i} (d\u) - d\u) v_{K_i} \dx  \\
&=: I_1 + I_2 + I_3.
\end{align*}

We first focus on the term $I_1$. Since the function $v_{K_i}$ vanishes at the endpoints of $K_i$, we may view it, by extension by zero outside of~$K_i$, as an element of $H^1_0(\Omega)$. We observe 
$$
I_1 = a(u,v_{K_i}) - a(\u,v_{K_i}) = a(u - \u,v_{K_i}) \leq \NNN{u -\u}_{K_i} \NNN{v_{K_i}}_{K_i},
$$
where the subscript $K_i$ in the norms indicates that the defining integral is taken over $K_i$ and not over $\Omega$. 

We now claim that, for $\beta \in (\nicefrac12,1]$, we have 
\begin{equation}
\label{eq:lemma:MW-lemma-3.3-10}
\NNN{v_{K_i}}_{K_i} \leq C \left[ \sqrt\e p_i^{2-\beta} h_i^{-1} + \left\|\sqrt{|d|}\right\|_{L^\infty(K_i)} \right]\left\|\Phi^{-\nicefrac{\beta}{2}} v_{K_i}\right\|_{L^2(K_i)}. 
\end{equation}
To see this, we compute with the product rule 
\begin{align*}
\|v_{K_i}^\prime\|_{L^2(K_i)} &\lesssim 
\left\|\Phi_{K_i}^{\beta} (\Pi_{K_i} f - (-\e \u^{\prime\prime} + \Pi_{K_i} (d\u)))^\prime\right\|_{L^2(K_i)}\\
&\quad + 
h_{i}^{-1} \left\|\Phi_{K_i}^{\beta-1} (\Pi_{K_i} f - (-\e \u^{\prime\prime} + \Pi_{K_i} (d\u))\right\|_{L^2(K_i)},
\end{align*}
and use the fact that $\Pi_{K_i}f - (-\e \u^{\prime\prime} + \Pi_{K_i} (d\u))$ is a polynomial: 
For the first term, we employ \cite[Lemma~{2.4}, $3^{rd}$ estimate]{MelenkAPOST01}, and 
for the second term, we apply \cite[Lemma~{2.4}, $2^{nd}$ estimate]{MelenkAPOST01} (this is the point where we need $\beta > \nicefrac12$ so that $2(\beta-1) > -1$) to get 
\begin{align*}
\Big\|\Phi_{K_i}^{\beta} (&\Pi_{K_i} f - (-\e \u^{\prime\prime} + \Pi_{K_i} (d\u)))^\prime\Big\|_{L^2(K_i)} \\
& \lesssim 
p_i^{2-\beta} h_i^{-1}\left\|\Phi_{K_i}^{\nicefrac{\beta}{2}} (\Pi_{K_i} f - (-\e \u^{\prime\prime} + \Pi_{K_i} (d\u)))\right\|_{L^2(K_i)}\\
& = p_i^{2-\beta} h_i^{-1}\left\|\Phi_{K_i}^{-\nicefrac{\beta}{2}} v_{K_i}\right\|_{L^2(K_i)}, 
 \end{align*}
 and analogously,
 \begin{align*}
h_{i}^{-1} \Big\|\Phi_{K_i}^{\beta-1} (\Pi_{K_i} f - (-\e \u^{\prime\prime} + \Pi_{K_i} (d\u)))\Big\|_{L^2(K_i)} 
& \lesssim 
%p^{2-\beta} h_i^{-1}\left\|\Phi_{K_i}^{\nicefrac{\beta}{2}} (\Pi_{K_i} f - (-\e \u^{\prime\prime} + \Pi_{K_i} (d\u)))\right\|_{L^2(K_i)} 
 p_i^{2-\beta} h_i^{-1}\left\|\Phi_{K_i}^{-\nicefrac{\beta}{2}} v_{K_i}\right\|_{L^2(K_i)}. 
\end{align*}
Furthermore, we note the simple estimate 
$$
\left\| \sqrt{|d|} v_{K_i} \right\|_{L^2(K_i)} \leq \left\|\sqrt{|d|} \right\|_{L^\infty(K_i)} \left\|\Phi_{K_i}^{-\nicefrac{\beta}{2}} v_{K_i}\right\|_{L^2(K_i)}. 
$$
It follows that
\[
\NNN{v_{K_i}}_{K_i} \lesssim \left( \sqrt{\e} p_i^{2-\beta} h_i^{-1} + \left\|\sqrt{|d|}\right\|_{L^\infty(K_i)} \right) \left\|\Phi_{K_i}^{-\nicefrac{\beta}{2}} v_{K_i}\right\|_{L^2(K_i)}, 
\]
which is the claimed estimate \eqref{eq:lemma:MW-lemma-3.3-10}. 

The terms $I_2$ and $I_3$ are estimated straightforwardly by 
$$
|I_2| + |I_3| \leq 
\left( \left\|\Phi_{K_i}^{\nicefrac{\beta}{2}} (f - \Pi_{K_i} f)\right\|_{L^2(K_i)} + 
\left\|\Phi_{K_i}^{\nicefrac{\beta}{2}} (d \u - \Pi_{K_i} (d \u))\right\|_{L^2(K_i)} \right) \left\|\Phi_{K_i}^{-\nicefrac{\beta}{2}} v_{K_i}\right\|_{L^2(K_i)}. 
$$

We conclude for any $\beta \in (\nicefrac12,1]$ the existence of a constant~$C > 0$ (depending only on $\beta$) such that 
\begin{equation}
\label{eq:lemma:MW-lemma-3.3-20}
\begin{split}
\left\|\Phi_{K_i}^{-\nicefrac{\beta}{2}} v_{K_i} \right\|_{L^2(K_i)} 
&\lesssim  
\left( \sqrt{\e} p_i^{2-\beta} h_i^{-1} + \left\|\sqrt{|d|}\right\|_{L^\infty(K_i)} \right) \NNN{u - \u}_{K_i}\\
&\quad + 
 \left\|\Phi_{K_i}^{\nicefrac{\beta}{2}} (f - \Pi_{K_i} f)\right\|_{L^2(K_i)} + 
\left\|\Phi_{K_i}^{\nicefrac{\beta}{2}} (d \u - \Pi_{K_i} (d \u))\right\|_{L^2(K_i)}.
\end{split}
\end{equation}

We now turn to bounding the volume contribution of the {\sl a posteriori}  error estimator. 
We fix $\beta \in (\nicefrac12,1]$, and estimate with the aid of~\cite[Lemma~{2.4}, $2^{nd}$ estimate]{MelenkAPOST01}:
\begin{align*}
\|f& - (-\e \u^{\prime\prime} + d \u)\|_{L^2(K_i)} \\
&\leq 
\left\|\Pi_{K_i} f - (-\e \u^{\prime\prime} + \Pi_{K_i} (d \u))\right\|_{L^2(K_i)} + 
\|f - \Pi_{K_i} f\|_{L^2(K_i)} + 
\|d \u - \Pi_{K_i} (d \u)\|_{L^2(K_i)}  \\
&\lesssim p_i^{\beta} \left\|\Phi_{K_i}^{\nicefrac{\beta}{2}} \left( \Pi_{K_i} f - (-\e \u^{\prime\prime} + \Pi_{K_i} (d \u))\right)\right\|_{L^2(K_i)}  \\
&\quad+\|f - \Pi_{K_i} f\|_{L^2(K_i)} + 
\|d \u - \Pi_{K_i} (d \u)\|_{L^2(K_i)}  \\
& = p_i^{\beta}  \|\Phi_{K_i}^{-\nicefrac{\beta}{2}} v_{K_i}\|_{L^2(K_i)}  + 
\|f - \Pi_{K_i} f\|_{L^2(K_i)} + 
\|d \u - \Pi_{K_i} (d \u)\|_{L^2(K_i)}, 
\end{align*}
which, recalling~\eqref{eq:lemma:MW-lemma-3.3-20}, results in (with implied constant depending on $\beta \in (\nicefrac12,1]$)
\begin{align*}
\|f - (-\e \u^{\prime\prime} + d \u)\|_{L^2(K_i)} 
&\lesssim 
\left(
\sqrt{\e} \frac{p_i^2}{h_i} + p_i^\beta \left\|\sqrt{|d|}\right\|_{L^\infty(K_i)}
\right) \NNN{u - \u}_{K_i} \\
& \quad + 
p_i^\beta
\left[ 
\left\|\Phi_{K_i}^{\nicefrac{\beta}{2}} (f - \Pi_{K_i} f)\right\|_{L^2(K_i)} + 
\left\|\Phi_{K_i}^{\nicefrac{\beta}{2}} (d \u - \Pi_{K_i} (d \u))\right\|_{L^2(K_i)} 
\right] \\
&\quad 
+ 
\|f - \Pi_{K_i} f\|_{L^2(K_i)} + 
\|d \u - \Pi_{K_i} (d \u)\|_{L^2(K_i)} . 
\end{align*}
This completes the proof.
\end{proof}

%For the control of the jump term, we need an observation that simplifies the parameter 
%$\gamma_i$ in~\eqref{eq:gamma}: 
%\begin{lemma}
%\label{lemma:estimate-of-gamma}
%Let $\alpha_j$, $\beta_j$ and $\gamma_j$ be defined as in Theorem~\ref{thm:main}. Then 
%\begin{equation}
%\label{eq:lemma:estimate-of-gamma-1}
%\sqrt{\varepsilon^{-1} \alpha_j} \leq \beta_j  = h_j^{-1}\alpha_j  + \sqrt{\varepsilon^{-1} \alpha_j} \leq 3 \sqrt{\varepsilon^{-1} \alpha_j}. 
%\end{equation}
%\end{lemma}
%\begin{proof}
%
%The first inequality in \eqref{eq:lemma:estimate-of-gamma-1} is evident. The second one follows
%from the observation 
%$$
%p_j^2 \left( h_j^{-1} \alpha_j \frac{1}{\sqrt{\varepsilon^{-1} \alpha_j}}\right)^2
% = 
%p_j^2 \frac{\varepsilon \alpha_j}{h_j^2}
%\leq 1
%$$
%so that 
%$$
%\beta_j = h_j^{-1} \alpha_j + 2 \sqrt{\varepsilon^{-1} \alpha_j} 
%\leq \left(2 + p_j^{-1}\right) \sqrt{\varepsilon^{-1} \alpha_j}. 
%$$
%\end{proof}

For every interior node $x_i$, $i=1,\ldots,N-1$, let $\omega_i:= (x_{i-1},x_{i+1})$ 
be the node patch associated with node $x_i$.
 
\begin{lemma}
\label{lemma:efficiency-jump-term}
Let $x_i$ be an interior node with node patch $\omega_i$, and~$\u$ the $hp$-FEM solution from~\eqref{eq:hpFEM}. For any $\delta_i>0$, there   holds 
\begin{equation}\label{eq:lemmaB3}
\e |\jmp{\u^\prime}(x_i)| 
\leq \left( \left(\e\delta_i^{-1}\right)^{\nicefrac12} + \delta_i^{\nicefrac12} \left\|\sqrt{|d|}\right\|_{L^\infty(\omega_i)}\right)\NNN{u - \u}_{\omega_i}
+ \delta_i^{\nicefrac12} \|r_i\|_{L^2(\omega_i)},
\end{equation}
where~$r_i:= f - (-\e \u^{\prime\prime} + d \u)$ is a function defined on~$\omega_i\setminus\{x_i\}$. 
\end{lemma}

\begin{proof}
Let $x_i$ be an interior node, and~$\delta_i>0$. Moreover, let $\psi_i\in H^1_0(\omega_i)$ be a cut-off function with~$\psi_i(x_i)=1$, and
\[
\|\psi_i\|_{L^2(\omega_i)} \leq \delta_i^{\nicefrac12}, 
\qquad 
\|\psi^\prime_i\|_{L^2(\omega_i)} \leq \delta_i^{-\nicefrac12}. 
\]
Then, the function $\widetilde \psi_i:= \jmp{\u^\prime}(x_i) \psi_i$ belongs to $H^1_0(\omega_i)$ (and is extended by zero to yield a function in~$H^1_0(\Omega)$). An integration by parts gives 
\begin{align*}
\e |\jmp{\u^\prime}(x_i)|^2 = \e \jmp{\u^\prime}(x_i) \widetilde \psi_i(x_i) 
&= - \int_{\omega_i} \e \u^\prime \widetilde \psi_i^\prime\,\dx  
- \int_{\omega_i} \e\u^{\prime\prime} \widetilde \psi_i\,\dx \\
&= a(u - \u,\widetilde \psi_i) - \int_{\omega_i} r_i \widetilde \psi_i\,\dx, 
\end{align*}
and thus 
\begin{align*}
\e |\jmp{\u^\prime}(x_i)|^2 &\leq |\jmp{\u^\prime}(x_i)| 
\left[ \NNN{u - \u}_{\omega_i}  \NNN{\psi_i}_{\omega_i} + 
\|r_i\|_{L^2(\omega_i)} \|\psi_i\|_{L^2(\omega_i)}
\right]. 
\end{align*}
We conclude with the properties of $\psi_i$: 
\begin{align*}
\e |\jmp{\u^\prime}(x_i)| 
\leq \left( \left(\e\delta_i^{-1}\right)^{\nicefrac12} + \delta_i^{\nicefrac12} \left\|\sqrt{|d|}\right\|_{L^\infty(\omega_i)}\right)\NNN{u - \u}_{\omega_i}
+ \delta_i^{\nicefrac12} \|r_i\|_{L^2(\omega_i)},
\end{align*}
which is the asserted estimate.
\end{proof}

As already mentioned in Remark~\ref{rem:efficiency-volume-term}, a particularly good setting for efficiency estimates is that the coefficient function $d$ is bounded from below. 

\begin{theorem}\label{thm:efficiency-appendix}
Suppose that there exist constants~$0<d_0\le d_1<\infty$ with
\begin{equation*}
%\label{eq:appendix-conditions-on-d}
d_0 \leq \inf_{x \in \Omega} d(x) \leq \sup_{x \in \Omega} d(x) \leq d_1. 
\end{equation*}
Fix $\beta \in (\nicefrac12,1]$.  Then there exists a constant $C > 0$ (depending only on $\beta$, the ratio~$\nicefrac{d_1}{d_0}$, and the shape-regularity parameter $\mu$ from~\eqref{eq:gamma-shape-regular}) such that the following is true:
\begin{enumerate}[(i)]
\item \label{item:i}
Let $K_i$, $1\le i\le N$, be an element. Then, 
\begin{align*}
\alpha_i \|f - (-\e \u^{\prime\prime} + d \u)\|^2_{L^2(K_i)} &
\leq C \left[ p_i^2 \NNN{u - \u}^2_{K_i} + \alpha_i R^2_{K_i}\right], 
\end{align*}
where we set  
\begin{align*}
R_{K_i} &= 
{p_i^\beta}
\left[ 
\left\|\Phi_{K_i}^{\nicefrac{\beta}{2}} (f - \Pi_{K_i} f)\right\|_{L^2(K_i)} + 
\left\|\Phi_{K_i}^{\nicefrac{\beta}{2}} (d \u - \Pi_{K_i} (d \u))\right\|_{L^2(K_i)} 
\right] \\
&\quad+ 
\|f - \Pi_{K_i} f\|_{L^2(K_i)} + 
\|d \u - \Pi_{K_i} (d \u)\|_{L^2(K_i)} .
\end{align*}
\item 
\label{item:thm:efficiency-appendix-ii}
Let $\omega_i = K_{i} \cup K_{i+1} \cup \{x_i\}$ be the node patch associated with the interior node~$x_i$, $1\le i\le N-1$. Then 
\begin{align*}
\gamma_i \e^2 |\jmp{\u}(x_i)|^2 
\leq C \left[ p^2_{i} \NNN{u - \u}^2_{\omega_i} 
+ \alpha_{i} R^2_{K_i} 
+ \alpha_{i+1} R^2_{K_{i+1}} 
\right]. 
\end{align*}
\end{enumerate}
\end{theorem}

\begin{proof} 
The estimate in~\eqref{item:i} follows directly from Lemma~\ref{lemma:MW-lemma-3.3} 
and the observation \eqref{eq:foo-10}. 
For \eqref{item:thm:efficiency-appendix-ii} we employ Lemma~\ref{lemma:efficiency-jump-term}. 
Let $K_i$, $K_{i+1}$ be the two elements sharing node $x_i$. By the shape regularity property~\eqref{eq:gamma-shape-regular}, and recalling Remark~\ref{rm:constants} we see that
$$
\gamma_i \sim \sqrt{\e^{-1} \alpha_{i}} \sim \sqrt{\e^{-1} \alpha_{i+1}}. 
$$
We will simply write $\alpha$ for $\alpha_{i}$ and $\gamma$ for $\gamma_i$. 
We make use of the freedom to select 
$
\delta_i:= \nicefrac{\alpha}{\gamma}
$
in Lemma~\ref{lemma:efficiency-jump-term}.
Then, employing~\eqref{eq:lemmaB3} and involving~$\gamma \sim \sqrt{\e^{-1} \alpha}$, we arrive at
\begin{align*}
\gamma \e^2 | \jmp{\u^\prime}(x_i)|^2 &\lesssim 
\gamma \left( \frac{\e}{\delta_i} + \delta_i\|d\|_{L^\infty(\omega_i)}\right)\NNN{u - \u}^2_{\omega_i} + \gamma \delta_i \|r_i\|^2_{L^2(\omega_i)} \\
& = \left (\frac{\gamma^2}{\alpha}\e + \alpha\|d\|_{L^\infty(\omega_i)} \right)\NNN{u - \u}^2_{\omega_i} + \alpha \|r_i\|^2_{L^2(\omega_i)} \\
&\lesssim (1 + \nicefrac{d_1}{d_0}) \NNN{u - \u}^2_{\omega_i} + \alpha \|r_i\|^2_{L^2(\omega_i)}. 
\end{align*}
We close the proof by remarking that the term $\alpha \|r_i\|^2_{L^2(\omega_i)}$ has been estimated earlier in \eqref{item:i}.
\end{proof}

\clearpage

\bibliographystyle{amsplain} \bibliography{literature}
%\renewcommand{\v}{}
%$$
%\v
%$$
\end{document}